\documentclass[12pt]{amsart}

\usepackage{amssymb, tikz}

\usepackage{hyperref}

\usepackage{amsmath}

\usepackage{amsmath, amsthm, amscd, amsfonts}

\usepackage{mathtools}

\usepackage[english]{babel}
\usepackage{setspace}
\usepackage[utf8]{inputenc}

\usepackage{mathrsfs}

\newtheorem{theorem}{Theorem}[section]

\newtheorem{lemma}[theorem]{Lemma}

\theoremstyle{definition}
\newtheorem{definition}[theorem]{Definition}

\theoremstyle{remark}

\newtheorem{notation}[theorem]{Notation}

\newtheorem{hypotheses}[theorem]{Hypotheses}

\newcommand{\dom}{{\rm dom}}

\newcommand{\cH}{{\mathscr H}}

\newcommand{\cI}{{\mathscr I}}

\newcommand{\MPB}{{\mathbb P}}

\newcommand{\Gen}{{\rm Gen}}

\newcommand{\CH}{{\rm CH}}

\newcommand{\GCH}{{\rm GCH}}
\def\MPB{{\mathbb{P}}}
\def\MQB{{\mathbb{Q}}}
\def\MRB{{\mathbb{R}}}

\newcount\skewfactor
\def\mathunderaccent#1#2 {\let\theaccent#1\skewfactor#2
\mathpalette\putaccentunder}
\def\putaccentunder#1#2{\oalign{$#1#2$\crcr\hidewidth
\vbox to.2ex{\hbox{$#1\skew\skewfactor\theaccent{}$}\vss}\hidewidth}}

\usepackage{hyperref}

\begin{document}

\title {The Abraham-Shelah $\Delta^2_2$-well-ordering of the reals}

\author[M.  Golshani]{Mohammad Golshani}

\address{Mohammad Golshani, School of Mathematics, Institute for Research in Fundamental Sciences (IPM), P.O.\ Box:
	19395--5746, Tehran, Iran.}

\email{golshani.m@gmail.com}
\urladdr{http://math.ipm.ac.ir/~golshani/}

\thanks{ The  author's research has been supported by a grant from IPM (No. 1401030417).}




\begin{abstract}
We give an exposition of the Abraham-Shelah's proof of the consistency of $\GCH$ with existence of a  $\Delta^2_2$-well-ordering of the reals.
\end{abstract}

\maketitle
\numberwithin{equation}{section}
\section{introduction}

In \cite{Ab-Sh}, Abraham and Shelah proved the following theorem.
\begin{theorem}
\label{th1}
Assume $\GCH$ holds. Then there exists a generic extension of the universe in which $\GCH$ hold and there exists a $\Delta^2_2$ well-ordering of the reals.
\end{theorem}
In this short note, we give an exposition of their proof, by giving some details for the missing parts.

The paper is organized as follows. In Section \ref{s-st-special} we present some preliminaries about proper forcing notions,  trees
and the Magidor-Malitz logic. In Section \ref{g15}, we show how to  specialize an Aronszajn tree on a stationary  subset of $\omega_1$. In Section \ref{g18} we prove an encoding theorem which plays a key role in the proof, and finally in
Section \ref{g19} we complete the proof of
theorem \ref{th1}.

\section{Some preliminaries}
\label{s-st-special}
In this section we provide some definitions and results which we will use for the rest of the paper.
\begin{hypotheses}
Through this paper we always assume that:
\begin{enumerate}
\item $\text{GCH}$ holds,
\item $S_* \subseteq \omega_1$ is stationary,
\end{enumerate}
\end{hypotheses}
\subsection{Properness and adding no reals}
\label{su1}
 We assume familiarity with proper forcing
and countable support iterations. For a forcing notion $\MPB$ and conditions $p, q \in \MPB,$ we say $q$ is stronger than $p$ if $q \geq p$.
We start by defining two notions which guarantee that a countable support iteration of forcing notions satisfying them does not add reals.
The first one is the notion of $<\omega_1$-properness.
\begin{definition}
\label{y12}
\begin{enumerate}
\item $\MPB$ is $\alpha$-proper
  if whenever $\chi$ is large enough regular, $\bar N = \langle N_i:i \le \alpha \rangle$ is an
increasing and continuous chain of countable elementary submodels of $({\cH}(\chi),\in)$ with $\alpha, \MPB \in N_0$
and
$\bar N \restriction (i+1) \in N_{i+1}$,
if $p \in \MPB \cap N_0$,
then there is $q,~p \le q \in \MPB$ such that $q$ is
$(N_i,\MPB)$-generic for each $i \le \alpha$.

\item We say $\MPB$ is $< \omega_1$-proper if $\MPB$ is
$\alpha$-proper for any $\alpha < \omega_1$.

\end{enumerate}
\end{definition}
The second notion is that of $\mathbb{D}$-completeness, that we give its definition is the sequel.
Let us start by fixing some notation.
\begin{notation}
\label{genPN}
Suppose $\MPB$ is a forcing notion, $p \in \MPB$ and  $N$ is a model with $\MPB \in N$. Then
\begin{enumerate}
\item[] $\Gen(N, \MPB) =\{ \bold G \subseteq \MPB \cap N: \bold G \text{~is a~} \MPB \cap N\text{-generic filter over~}N          \}.$

\item[] $\Gen^+(N, \MPB) =\{ \bold G \in \Gen(N, \MPB): G$ has an upper bound in $\MPB\}$.

\item[] $\Gen(N, \MPB, p) =\{ \bold G \in \Gen(N, \MPB): p \in \bold G\}$.
\end{enumerate}
\end{notation}

\begin{definition}
\label{completeness system}
\begin{enumerate}
\item (\cite[Ch. V, Definitions 5.2]{Sh:f}) A \emph{completeness system} for a forcing notion $\MPB$ is a function $\mathbb{D}$ such that the following statements hold:
\begin{enumerate}
\item For a sufficiently large $\theta,$ the domain of $\mathbb{D}$ consists of pairs $(N, p),$ where
$N \prec (H(\theta), \in)$ is countable, $\MPB \in N$ and $p \in \MPB \cap N,$

\item For every $(N, p) \in \dom(\mathbb{D}),$ $\mathbb{D}(N, p)$ is a collection of subsets of $\Gen(N, \MPB, p)$.
\end{enumerate}

\item (\cite[Ch. V, Definition 5.2]{Sh:f})
Suppose $\kappa$ is a cardinal. We say  $\mathbb{D}$ is a $\kappa$-completeness system for  $\MPB$, if it is a  completeness system
for $\MPB$ and for every $(N, p) \in \dom(\mathbb{D}),$ the intersection of fewer than $1+\kappa$ elements of $\mathbb{D}(N, p)$ is nonempty.

\item (\cite[Ch. V, Definition 5.4]{Sh:f})
A completeness system $\mathbb{D}$ for  $\MPB$ is \emph{simple}
if there is a second order formula $\Psi$ such that $\mathbb{D}(N, p)=\{\mathcal{G}_X: X \subseteq N          \}$, where
\[
\mathcal{G}_X=\{\bold G \in \Gen(N, \MPB, p): (N, \in, \MPB \cap N)\models \Psi(\bold G, X)    \}.
\]

\item (\cite[Ch. V, Definition 5.3]{Sh:f})
Suppose $\mathbb{D}$ is a simple completeness system for $\MPB$. Then $\MPB$ is said to be $\mathbb{D}$-complete, if for  every
 $(N, p) \in \dom(\mathbb{D}),$ $\Gen^+(N, \MPB, p)$ contains an element of $\mathbb{D}(N, p)$.
 \end{enumerate}
\end{definition}
The next theorem  gives a sufficient condition  for a countable support iteration of
forcing notions to not add new reals.
\begin{theorem}
\label{pre22}
(\cite[Ch. VIII, Theorem 4.5]{Sh:f})
A countable support iteration of forcing notions which are $<\omega_1$-proper and $\mathbb{D}$-complete
with respect to a simple 2-completeness system does not introduce reals.
\end{theorem}
We now introduce another notion which we will use to show that a countable support iteration of forcing notions satisfies the $\aleph_2$-c.c.
\begin{definition}
\label{picdef} (\cite[Ch. VIII, Definition 2.1]{Sh:f})
The forcing notion $\MPB$  satisfies  the $\kappa$-p.i.c ($\kappa$-properness isomorphism condition), if the following holds for any large enough regular cardinal $\lambda$:
Suppose $i< j<\kappa,  N_i, N_j \prec (\cH(\lambda), \in, \lhd_\lambda)$  (where $\lhd_\lambda$ is a well-ordering of $\cH(\lambda)$)
are countable such that $\kappa, \MPB \in N_i \cap N_j$, $i\in N_i, j \in N_j, N_i \cap \kappa \subseteq j, N_i \cap i = N_j \cap j, p \in N_i \cap \MPB$ and $h: N_i \cong N_j$ is such that $h \restriction N_i \cap N_j$ is identity and $h(i)=j.$
Then there exists $q \in \MPB$ such that:
\begin{itemize}
\item $q \geq p, h(p)$ and for every maximal antichain $\cI \in N_i$ of $\MPB$, we have that $\cI \cap N_i$
is predense above $q$ and similarly for $\cI \in N_j$,

\item for every $r \in N_i \cap \MPB$ and $q'\geq q,$ there is $q'' \geq q'$ such that $$r \leq q'' \iff h(r) \leq q''.$$
\end{itemize}
\end{definition}

\begin{theorem} ( \cite[Ch.VIII, Lemma 2.4]{Sh:f})
\label{ecc and pic and chain condition}
Assume CH holds.
If $\MPB$ is a countable support iteration of length at most $\omega_2$ whose iterands satisfy the $\aleph_2$-p.i.c, then $\MPB$
satisfies the $\aleph_2$-c.c.
\end{theorem}

\subsection{$S$-st-special  trees.}
In order to show that the Sosulin hypothesis does not imply all Aronszajn trees are special, Shelah \cite{shelah81} introduced several types of specialization of Aronszajn trees which are weaker that usual specialization but still strong enough to imply that the trees are not Souslin.
Here we just consider a simple version of such specialization and refer to \cite{shelah81} and \cite{Sh:f}  for more details.

By an $\aleph_1$-tree we mean a tree of height $\omega_1$ all of whose levels are countable and such that every node has countably many successors.
\begin{notation}
If $T$ is an $\aleph_1$-tree, $s \in T$ and $\alpha < \omega_1$,  then
\begin{itemize}
\item $T_s=\{t \in T: s \leq_T t  \}$.

\item $T_\alpha=\{s \in T: ht_T(s)=\alpha     \}$ is the $\alpha$-th level of $T$.

\item If $\beta \leq ht_T(s)$, then $s \upharpoonright \beta$ is the unique element of $T_\beta$ such
that $s \upharpoonright \beta \leq_T s.$ In general given a sequence $\vec{s}= \langle  s_i: i<n        \rangle$
of elements of $T$ and some $\beta \leq \min\{ht_T(s_i): i<n\}$ we define
\[
\vec{s} \upharpoonright \beta = \langle s_i \upharpoonright \beta: i<n           \rangle.
\]

\item If $S \subseteq \omega_1,$ then
\[
T \upharpoonright S= \{t \in T: ht_T(t) \in S    \}.
\]
 \end{itemize}
 \end{notation}
\begin{definition}
\label{s-special trees}
Suppose  $S \subseteq \omega_1$ is stationary and
 $T$ is an $\aleph_1$-tree.
  $T$ is $S$-st-special when there exists an $S$-specializing function $c$ of $T$ which means
\begin{enumerate}
\item $c:T \upharpoonright S \to \omega_1$,
\item If $t \in T_\delta$ where $\delta \in S,$ then $c(t)< \delta$,
\item If $s \neq t$ are in $T \upharpoonright S$ and $c(s)=c(t)$, then $s$ and $t$ are $<_T$-incomparable.
\end{enumerate}
\end{definition}
The next lemma is immediate using Fodor's lemma.
\begin{lemma}
\label{s-specialimpliesnosuslin}
Suppose $S$ is a stationary subset of $\omega_1$ and $T$ is $S$-st-special. Then $T$ has no cofinal
branches, in particular $T$ is an Aronszajn tree. Furthermore $T$ is not Souslin
\end{lemma}
\begin{proof}
Let $c: T \upharpoonright S \to \omega_1$ witness that $T$ is $S$-st-special.
First let us show that $T$ has no cofinal branches. Assume on the contrary that $b$ is a cofinal branch of $T$ and for each
$\alpha < \omega_1$ let $b(\alpha)$ be the node in $b \cap T_\alpha$, which is unique. We define $f: S \to \omega_1$
by $f(\alpha)=c(b_\alpha)$. Then $f$ is progressive, so by Fodor's lemma it is constant on a stationary subset $S'$ of $S$. But if
$\alpha < \beta$ are in $S'$, we have $f(\alpha)=c(b_\alpha) \neq c(b_\beta)=f(\beta),$ a contradiction.

To show that $T$ is not Souslin, we argue in a similar way. For each $\alpha \in S$ pick a node $b(\alpha) \in T_\alpha$. By the
above argument, there exists a  stationary subset $S'$ of $S$ such that for all $\alpha < \beta$ in $S'$, $c(b_\alpha) = c(b_\beta)$.
It follows that $\{b(\alpha): \alpha \in S'   \}$ is an antichain of $T$ of size $\aleph_1$.
\end{proof}
We now define product and disjoint union of trees.
\begin{definition}
\label{f1}
Suppose $n<\omega$ and $T_0, \cdots, T_{n-1}$ are $\aleph_1$-trees.
\begin{enumerate}
\item $\bigotimes_{i<n}T_i=\{\bar t \in \prod_{i<n}T_i: \exists \alpha<\omega_1 \forall i<n, ht_{T_i}(t_i) =\alpha              \},$

\item Suppose $T_i$'s are pairwise disjoint. Then $\bigoplus_{i<n}T_i=\bigcup_{i<n}T_i$.
\end{enumerate}
\end{definition}
We can define $\bigoplus_{i<n}T_i$ in general by first making the $T_i$'s disjoint and then taking their union.
\begin{definition}
\label{f2}
Suppose $T$ is an $\aleph_1$ tree.
 A derived tree of $T$ is a tree of the form $T_{\vec{s}}=\bigotimes_{i<n}T_{s_i}$ where $\vec{s}= \langle s_i: i<n \rangle$ and for some $\alpha < \omega$ and each
$i<n, $ $s_i \in T_\alpha.$ 
\end{definition}
It is easily seen that a derived tree $T_{\vec{s}}$ of $T$ is Aronszajn iff for some $i<n, T_{s_i}$ is Aronszajn.
The next lemma gives a preservation result about Souslin trees.
\begin{lemma} (\cite[Theorem 3.1]{Ab-Sh}, \cite[Lemma 1.2]{miyamoto})
\label{f8}
Let $U$ be a Souslin tree. Then the property of a forcing poset being proper and forcing that $U$ is Souslin is preserved by any countable support forcing iteration.
\end{lemma}
The following lemma gives a characterization of when a Souslin tree kills an Aronszajn tree.
\begin{lemma} (\cite{lind})
\label{f9}
Let $U$ be a normal Souslin tree and $T$ a normal Aronszajn tree. Then $\Vdash_{U}$`` $T$ has a cofinal branch'' iff there exists a club $C \subseteq \omega_1$ and a strictly increasing and height preserving function $f: U \upharpoonright C \to T \upharpoonright C.$
\end{lemma}

\subsection{The Magidor-Malitz logic}
In \cite{magidor}, Magidor and Malitz introduced a new logic, denoted $L(Q^{MM})$, and studied some of its properties. It is obtained by adjoining to the first order logic the quantifiers $Qxy\phi(x, y)$ which is true in a structure if and only if there exists an uncountable subset
of that structure's universe such that for any two distinct $x$ and $y$ in the set, $\phi(x, y)$ holds. For the purpose of this paper, we will need the following.
\begin{lemma} (\cite{Ab-Sh})
\label{g12} For any formula $\Psi$ in the Magidor-Malitz logic $L(Q^{MM})$, the statement:
\begin{center}
``there is a model $\bold K$ of $\Psi$''
\end{center}
is equivalent to a $\Sigma^2_2$ statement.
\end{lemma}

\section{Specializing Aronszajn trees}
\label{g15}
In this section we prove the following main theorem which plays a key role in this paper.

\begin{theorem}
\label{k16}
Assume $S_* \subseteq \omega_1$ is stationary and $T$ is an Aronszajn tree. There exists a forcing notion
$\MPB_T$ such that:
\begin{enumerate}
\item $\MPB_T$ is $< \omega_1$-proper and $\mathbb{D}$-complete for an $\aleph_1$-completeness system
$\mathbb{D}$.

\item $\MPB_T$ satisfies the $\aleph_2$-p.i.c.,

\item Forcing with $\MPB_T$ adds no new reals,


\item $T$ becomes $S_*$-st-special in $V^{\MPB_T}$.
\end{enumerate}
\end{theorem}
The rest of this section is devoted to the proof of the above theorem. 
The forcing notion we use was first defined by Shelah in
\cite{shelah81} (see also \cite[Ch. IX]{Sh:f}). We follow Schlindwein's approach
from  \cite{chaz94}.

For $n<\omega$ let $T^n=\bigotimes_{i<n}T$. We first define an auxiliary forcing notion
$\MPB^0_T$.
\begin{definition}
\label{f3} A condition in $\MPB^0_T$ is a pair $p=(f_p, S_p)$, where:
\begin{enumerate}
\item $S_p \subseteq S_*$ and $cl(S_p) \cap S_* \subseteq S_p$, where $cl(S_p)$ stands for the closure of $S_p$ in $\omega_1$,

\item if $t \in \dom(f_p),$ then $f_p(t) <_T ht_T(t)$,

\item if $s, t \in \dom(f_p)$ and $f_p(s)=f_p(t)$, then $s$ and $t$ are $<_T$-incomparable.
 \end{enumerate}
$\MPB^0_T$ is ordered in the natural way: $p \leq q$ iff $f_{q} \supseteq f_{p}$ and $S_{q}$ end extends $S_p$.
\end{definition}
\begin{definition}
Given an Aronszajn tree $T$, $\vec s \in (T^n)_\beta$, $f: T \upharpoonright S \to \omega_1$
and $F \subseteq \omega_1$ finite, let $\heartsuit(\alpha, \vec s, f, F)$ stand for:
\begin{center}
 if $\xi \in S \cap (\alpha, \beta],$ then $f(s_i \restriction \xi) \notin F$ for all $i<n$.
\end{center}
\end{definition}
\begin{definition}
Suppose $T$ is an Aronszajn tree. We say $\Gamma$ is a promise, if for some club $C \subseteq \omega_1$
and some  $\vec s \in T^n,$
\begin{enumerate}
\item $\Gamma \subseteq T^n \restriction C,$

\item for all $\vec t \in \Gamma, \vec t \geq_{T^n} \vec s$,

\item if $\vec t \in \Gamma$ and $\alpha \in ht_{T^n}(\vec t) \cap C,$ then $\vec t \restriction \alpha \in \Gamma,$

\item suppose $\alpha < \beta$ are in $C$ and $\vec t \in \Gamma(\alpha)=\Gamma \cap (T^n)_\alpha$. Then there exists an infinite $W \subseteq \Gamma(\beta)$ such that:
    \begin{enumerate}
    \item for all $\vec u \in W, \vec u \geq_{T^n} \vec t$,

    \item for all $\vec u^1 \neq \vec u^2$ in $W$, $\vec u^1 \cap \vec u^2=\emptyset,$ in the sense that $\{ u^1_0, \cdots, u^1_{n-1}  \}   \cap \{u^2_0, \cdots, u^2_{n-1}  \}=\emptyset.$
    \end{enumerate}
\end{enumerate}
We write $C(\Gamma)=C, n(\Gamma)=n$ and $\min(\Gamma)=\vec{s}$.
\end{definition}
\begin{definition}
\label{f5} Assume $S_* \subseteq \omega_1$ is stationary, co-stationary and $T$ is an Aronszajn tree. Let also
$(f, S) \in \MPB^0_T$ and let $\Gamma$ be a promise for $T$. We say $(f, S)$ fulfills $\Gamma$
iff:
\begin{enumerate}
\item $S \setminus ht(\min(\Gamma)) \subseteq C(\Gamma)$,

\item Suppose $\beta \in C(\Gamma)$, $\alpha \in C(\Gamma) \cap S \cap \beta$, $\vec t \in \Gamma(\alpha)$
and $F \subseteq \omega_1$ is finite. Then there exists an infinite $W \subseteq \Gamma(\beta)$ such that
\begin{enumerate}
    \item for all $\vec u \in W, \vec u \geq_{T^n} \vec t$,

    \item for all $\vec u^1 \neq \vec u^2$ in $W$, $\vec u^1 \cap \vec u^2=\emptyset,$

    \item for all $\vec u \in W, \heartsuit(\alpha, \vec u, f, F)$ holds.
    \end{enumerate}
\end{enumerate}
\end{definition}
We are finally ready to define the forcing notion $\MPB_T.$
\begin{definition}
\label{f6} A condition in $\MPB_T$ is a triple $p=(f_p, S_p, \Psi_p)$, where
\begin{enumerate}
\item $(f_p, S_p) \in \MPB^0_T$,
\item $\Psi_p$ is a countable set of promises which $(f_p, S_p)$ fulfills.
\end{enumerate}
Given two conditions $p$ and $q$, let $p \leq q$
iff
\begin{enumerate}
\item $(f_p, S_p) \leq_{\MPB^0_T} (f_q, S_q)$,

\item $\Psi_q \supseteq \Psi_p$,

\item for all $\Gamma \in \Psi_p, S_q\setminus S_p \subseteq C(\Gamma)$.
\end{enumerate}
\end{definition}
By   \cite{chaz94}, the forcing notion
$\MPB_T$ satisfies items (1)-(4) of Theorem \ref{k16}.
The proof of the next lemma is similar to the proof of \cite[Theorem 4.6]{Ab-Sh}
\begin{lemma}
\label{g3} Suppose $S^* \subseteq \omega_1$ is stationary co-stationary, $U$ is a Souslin tree and $T$ is an Aronszajn tree.
If $\Vdash_{U}$``$T$ is Aronszajn'', then $\Vdash_{\MPB_T}$``$U$ is Souslin''.
\end{lemma}

\section{An encoding theorem}
\label{g18}
In this section we prove an encoding theorem (see Theorem \ref{k19}), which will be used in the next section for the proof of theorem \ref{th1}.
The next definition is an analogue of \cite[Definition 7.1]{Ab-Sh}, where instead of working with special trees we work with $S$-st-Special trees.
\begin{definition}
\label{k14} Suppose $S \subseteq \omega_1$ is stationary, $I$ is an $\omega_1$-like linear order \footnote{Here by an $\omega_1$-like linear order we man an uncountable order all of whose initial segments are countable and such that $I$ has  a fist element and every element has a successor.} and $(\bold{su}, \bold{sp})$
is a partition of $[I]^{<\omega}\setminus \{\emptyset\}$ such that $\bold{su}$ is closed under subsets. We say that an
$I$-sequence $\langle  T_i: i \in I  \rangle$ of $\aleph_1$-Aronszajn trees has the pattern
$S$-$(\bold{su}, \bold{sp})$ if:
\begin{enumerate}
\item for $d \in \bold{su}$, every derived tree of $\bigoplus_{i\in d}T_i$ is Souslin,

\item For $d \in \bold{sp}$, $\bigotimes_{i\in d}T_i$ is $S$-st-special.
\end{enumerate}
\end{definition}
The next lemma shows that $\Diamond_{\omega_1}$  guarantees the existence of $S$-$(\bold{su}, \bold{sp})$ patterns for all stationary sets
$S \subseteq \omega_1$ and suitable partitions $(\bold{su}, \bold{sp})$ of $[\omega_1]^{<\omega}\setminus \{\emptyset\}.$
\begin{lemma} (\cite[Theorem 2.2]{Ab-Sh})
\label{k12}  Assume $\Diamond_{\omega_1}$ holds. Let $\bold{sp}$ be a collection of non-empty finite
subsets of $\omega_1$ closed under supersets and let $\bold{su}$ be those non-empty finite subsets of $\omega_1$
which are not in $\bold{sp}$. Then there exists a sequence $\langle T^\xi: \xi < \omega_1    \rangle$
of $\aleph_1$-trees which is an $S$-$(\bold{su}, \bold{sp})$-pattern, indeed for each finite set $d=\{\xi_0, \cdots, \xi_{n-1}   \}$,
\begin{enumerate}
\item if $d \in \bold{su}$, then $\bigoplus_{i<n} T^{\xi_i}$ and all of its derived trees are Souslin,

\item if $d \in \bold{sp}$, then $\bigotimes_{i<n} T^{\xi_i}$ is special.
\end{enumerate}
\end{lemma}

\begin{definition} (\cite[Definition 7.2]{Ab-Sh})
\label{k15}
\begin{enumerate}
\item A collection $\mathcal U$ of Souslin trees is \emph{primal} if all derived trees of trees in
$\mathcal U$ are Souslin and for any Souslin tree $T$, there exists some $U \in \mathcal U$
such that a derived tree of $U$
is club embeddable into $T$.

\item  Suppose $S \subseteq \omega_1$ is stationary and $I$ is an $\omega_1$-like linear order. The $I$-sequence $\langle  T_i: i \in I  \rangle$ with pattern $S$-$(\bold{su}, \bold{sp})$ is called primal if the collection $\mathcal U=\{\bigoplus_{i\in d}T_i: d \in \bold{su}   \}$
    is primal.
\end{enumerate}
\end{definition}
\begin{theorem}
\label{k17} Assume $S^* \subseteq \omega_1$ is stationary and let $\mathcal U$ be a collection of Souslin trees such that
for all $U \in \mathcal U,$ all derived trees of $U$ are Souslin. Then there is a
forcing notion $\MPB=\MPB_{S^*, \mathcal{U}}$ such that:
\begin{enumerate}
\item $\MPB$ adds no new countable sequences and is $\aleph_2$-c.c.,


\item in $V^{\MPB}$, $\mathcal U$ is a primal collection of Souslin trees.

\end{enumerate}
\end{theorem}
\begin{proof}
The proof of the theorem is as in \cite[Theorem 7.3(2)]{Ab-Sh}, where instead of the forcing notions $\mathscr S(T)$ used there we use the forcing notions $\MPB_T$
of Theorem \ref{k16}. For completeness we give a proof.

Let $\Phi: \omega_2 \to \cH(\omega_2)$ be such that for each $x  \in \cH(\omega_2)$,
$\Phi^{-1}(x)\subseteq \omega_2$ is cofinal in $\omega_2$, whose existence follows from $\text{GCH}$.
Let
\[
\MPB=\langle \langle \MPB_\alpha: \alpha \leq \omega_2             \rangle, \langle  \dot{\MQB}_\alpha: \alpha<\omega_2       \rangle\rangle
\]
be a countable support iteration of forcing notions so that at stage $\alpha$ of the iteration we force with the trivial forcing notion  unless
the following conditions are satisfied:
\begin{enumerate}
\item $\Phi(\alpha)$ is a $\MPB_\alpha$-name of an Aronszan tree,
\item for any Souslin tree $U \in \mathcal{U}$ and any derived tree $U^\dag$ of $U$, in $V^{\MPB_\alpha}$,
$\Vdash_{U^\dag}$`` $\Phi(\alpha)$ is Aronszajn''.\footnote{Recall that by lemmas \ref{g3} and \ref{f8}, $U^\dag$ remains Souslin in $V^{\MPB_\alpha}$.}
\end{enumerate}
In this case we let $\dot{\MQB}_\alpha$ be a $\MPB_\alpha$-name such that $\Vdash_{\MPB_\alpha}$`` $\dot{\MQB}_\alpha=\MPB_{\Phi(\alpha)}$''.

By Theorem \ref{k16} and the results of subsection \ref{su1}, $\MPB$ satisfies item (1)  of the theorem.
For clause (2), first note that by Lemmas \ref{g3} and \ref{f8}, all trees in $\mathcal U$ and their derived trees remain Souslin in the generic extension
$V^{\MPB}$. Next suppose that $T \in V^{\MPB}$ is a Souslin tree and let $\dot{T}$ be a $\MPB$-name for it. Let $\alpha < \omega_2$ be such that $\dot{T}$ is a $\MPB_\alpha$-name and $\Vdash_{\MPB_\alpha}$`` $\dot T=\Phi(\alpha)$''. At stage $\alpha$ of the iteration, we should force with the trivial forcing, as otherwise
we will have $\Vdash_{\MPB_\alpha}$`` $\dot{\MQB}_\alpha=\MPB_{\dot{T}}$'', hence $T$ becomes $S$-st-special in
$V^{\MPB_{\alpha+1}}$ and hence in $V^{\MPB}$, which contradicts Lemma \ref{s-specialimpliesnosuslin}.
It follows that for some $U \in \mathcal U$ and some derived tree $U^\dag$ of $U$,
\[
V^{\MPB_\alpha} \models \ulcorner \Vdash_{U^\dag}\text{``~}\dot{T} \text{ is not Aronszajn''} \urcorner.
\]
We may assume that $U^\dag$ is of minimal dimension, so that it is a normal Souslin tree.
By Lemma \ref{f9}, $U^\dag$ is club embeddable into $T$\footnote{This holds in $V^{\MPB_\alpha}$ and hence also in $V^{\MPB}$.}.
\end{proof}

The proof of the following theorem is essentially the same as in \cite{Ab-Sh}, where instead of Theorem 7.3 from there we use Theorem \ref{k17} .
\begin{theorem} (Encoding theorem)
\label{k19} Assume $S^*$ is a stationary  subset of $\omega_1$ which only contains limit ordinals. There is a sentence $\Psi$
in the Magidor-Malitz logic which contains among other things a one-place predicate $I$ for an $\omega_1$-like linear order and  one-place predicates $P(x) $ and $S(x)$
such that the following holds. Given any $X \subseteq \omega_1$:
\begin{enumerate}
\item there is a model $M \models \Psi$ enriching $(\omega_1, <, X)$ such that  $P^M=X$
and $S^M=S^*,$

\item there is an $\aleph_2$-c.c. generic extension of the universe which adds no new countable sets and such that in it the following holds: if $N$ is a model of $\Psi$,
then $I^N$ has order type $\omega_1$, $P^N=X$ and $S^N=S^*,$

\item there is an $\aleph_2$-c.c. generic extension of the universe which adds no new countable sets, in which
 $M$ is up to isomorphism the only model of $\Psi$
\end{enumerate}
\end{theorem}
\begin{proof}
We start by describing the  sentence $\Psi$ (and its language).
 The sentence $\Psi$ describes the following:
\begin{itemize}
\item $(I, \prec)$ is an $\omega_1$-like order,
\item $\langle T^\xi: \xi \in I   \rangle$ is an $I$-sequence of $\omega_1$-like Aronszajn trees,
\item $S \subseteq I$ consists of limit points of $(I, \prec)$,
\item $(\bold{su}, \bold{sp})$ is a simple  pattern, as described in \cite[Subsection 7.1]{Ab-Sh},

\item for each $d \in \bold{su}, \bigoplus_{\xi \in d}T^\xi$ is an $\omega_1$-like Souslin tree,

\item for each  $d \in \bold{sp},~ \bigotimes_{\xi\in d}T^\xi$ is an $\omega_1$-like $S$-special Aronszajn  tree
as witnessed by $f(d)$,

\item $P, \tilde{P}$ are subsets of $I$,
\item $S$ is equal to  $\tilde{P} \cap \lim(I)$, where $\lim(I)$ is the set of limit points of $I$,
\item $\forall \xi \in I \big( P(\xi) \iff \tilde{P}(\xi+1) \big)$,
\item $\forall \xi \in I \big(  \langle  3, 5, 6, \xi     \rangle \in \bold{su} \iff \tilde{P}(\xi)        \big)$.
\end{itemize}
Now given any $X \subseteq \omega_1$ let
 \[
 \tilde{X}=S^* \cup \{\xi+1: \xi \in X      \}
 \]
 and let $(\bold{su}, \bold{sp})$ be a simple $S^*$-pattern with $\tilde{X}=\{ \xi < \omega_1:   \langle  3, 5, 6, \xi     \rangle \in \bold{su}     \}$.  By Lemma \ref{k12}, there exists a sequence $\langle T^\xi: \xi < \omega_1 \rangle$ of $\aleph_1$-Aronszajn trees which has the pattern $S^*$-$(\bold{su}, \bold{sp})$. Then for all $\xi < \omega_1,$
 \[
 \xi \in X \Leftrightarrow \xi+1 \in \tilde{X} \Leftrightarrow  \langle  3, 5, 6, \xi+1     \rangle \in \bold{su} \Leftrightarrow \tilde{P}(\xi+1) \Leftrightarrow P(\xi).
 \]
 Similarly for each limit ordinal $\xi < \omega_1$,
 \[
 \xi \in S^* \Leftrightarrow \xi\in \tilde{X} \Leftrightarrow  \langle  3, 5, 6, \xi     \rangle \in \bold{su} \Leftrightarrow \tilde{P}(\xi) \Leftrightarrow \xi \in S.
 \]
 This takes care of (1).

To prove (2), let $\mathcal{U}=\{\bigoplus_{\xi\in d}T^\xi: d \in \bold{su}   \},$ and let $\MPB=\MPB_{S^*, \mathcal{U}}$ be the forcing notion of Theorem \ref{k17}.
The forcing is $\aleph_2$-c.c.
Let $N$ be a model of $\Psi$. Then $I^N$ is an $\omega_1$-like order and the sequence $\langle  (T^\xi)^N: \xi \in I \rangle$
 has the simple pattern $S^*$-$(\bold{su}, \bold{sp})$. By the uniqueness of simple patterns \footnote{See \cite[Theorem 7.4]{Ab-Sh}. We may note that the theorem is stated for special trees, but its proof works for $S^*$-st-special trees as well.} $I^N$ is isomorphic to $\omega_1$,
 and after such an identification $(\bold{su}^N, \bold{sp}^N)=(\bold{su}, \bold{sp})$.  From this it follows that $P^N=X$ and $S^N=S^*.$

 Clause (3) can be proved as in \cite[Subsection 8.1]{Ab-Sh}, so we skip its proof.
\end{proof}
\section{Proof of main theorem}
\label{g19}
In this section we prove theorem \ref{th1}.
\\
{\bf Proof of Theorem \ref{th1}:}
In $V$, let $S^* \subseteq \omega_1$ be stationary such that
such that $\Diamond(S^*)$ holds.
Let $\Psi$ be the sentence in the  Magidor-Malitz logic with  a one-place predicates $P(x)$ and $S(x)$ given by Theorem \ref{k19}. Since $\CH$ holds in $V$, we can fin a subset $P \subseteq \omega_1$ which encodes in a natural way a well-order $\langle r_\alpha: \alpha < \omega_1      \rangle$ of $\MRB$ of order type $\omega_1.$  Now let $\phi$ be the sentence:
 \begin{center}
 ``there is a model $\bold K$ of $\Psi$ where $r_\alpha$ appears in $P^{\bold K}$ before $r_\beta$ does''.
 \end{center}
 By Lemma \ref{g12}, $\phi$ is a $\Sigma^2_2$-statement.
 By Theorem \ref{k19}, $\GCH$
is consistent with $\phi$. Indeed we can find a model $\bold M$ of $\Psi$ with $P^{\bold M}=P$ and a generic extension  in which $\bold M$ is the unique model of $\Psi.$
In this generic extension the relation $r_\alpha < r_\beta$ defined by
the above formula $\phi$
is $\Sigma^2_2$, and hence $\Delta^2_2$ (since any $\Sigma^2_2$ linear order is $\Delta^2_2$).
This gives the proof of Theorem \ref{th1}.

\end{document}